\documentclass[11pt]{amsart}
\usepackage[margin=1.15in]{geometry}
\usepackage{amscd,amssymb, amsmath, wasysym}
\usepackage{graphicx}
\usepackage{amsfonts}
\usepackage{mathrsfs}    
\usepackage{amsmath}    
\usepackage{amsthm}     
\usepackage{amscd}      
\usepackage{amssymb}    
\usepackage{eucal}      
\usepackage{latexsym}   
\usepackage{graphicx}   
\usepackage{verbatim}   
\usepackage[all]{xy}     

\makeatletter

\newcounter{thmcounter}

\numberwithin{thmcounter}{section}
\numberwithin{equation}{thmcounter}

\newtheorem{theorem}[thmcounter]{Theorem}
\newtheorem{proposition}[thmcounter]{Proposition}
\newtheorem{lemma}[thmcounter]{Lemma}
\newtheorem{corollary}[thmcounter]{Corollary}

\newtheorem*{Ein}{Ein's Lemma}

\theoremstyle{definition}
\newtheorem{definition}[thmcounter]{Definition}

\newtheorem{remark}[thmcounter]{Remark}

\newtheorem*{SitA}{Situation A}
\newtheorem*{SitB}{Situation B}
\newtheorem*{SitC}{Situation C}

\newtheoremstyle{claim}{9pt}{3pt}{}{\parindent}{\bf}{.}{1em}{}

\theoremstyle{claim}
\newtheorem{claim}[equation]{Claim}

\DeclareMathOperator{\tr}{tr}                  

\newenvironment{namelist}[1]{%
\begin{list}{}
{
\settowidth{\labelwidth}{#1}%
\setlength{\labelsep}{0.3em}%
\setlength{\leftmargin}{\labelwidth}%
\addtolength{\leftmargin}{\labelsep}}}{%
\end{list}}

\newcommand{\nR}{\mathbb{R}}                     
\newcommand{\nC}{\mathbb{C}}                     
\newcommand{\nQ}{\mathbb{Q}}                     

\newcommand{\nP}{\mathbb{P}}                     

\newcommand{\nA}{\mathbb{A}}                     

\newcommand{\sF}{\mathscr{F}}

\newcommand{\sO}{\mathscr{O}}                    
\newcommand{\sH}{\mathscr{H}}
\newcommand{\sI}{\mathscr{I}}                    

\newcommand{\sK}{\mathscr{K}}

\newcommand{\sQ}{\mathscr{Q}}

\newcommand{\I}{\mathscr{I}}

\newcommand{\mf}[1]{\mathfrak{#1}}

\DeclareMathOperator{\Bl}{Bl}                    
\DeclareMathOperator{\Bs}{Bs}                    

\DeclareMathOperator{\codim}{codim}              

\DeclareMathOperator{\lct}{lct}                  

\DeclareMathOperator{\Supp}{Supp}                
\DeclareMathOperator{\sHom}{\mathscr{H}om}       
\DeclareMathOperator{\Spec}{Spec}                
\DeclareMathOperator{\spec}{Spec}                

\DeclareMathOperator{\reg}{reg}                  
\DeclareMathOperator{\Exc}{Exc}                  

\newcommand{\ord}{\mbox{ord}}                    

\newcounter{rkcounter}             
\begin{document}

\title[Singularities of generic linkage]{Singularities of generic linkage of algebraic varieties}
\author{Wenbo Niu}

\address{Department of Mathematics, Purdue University, West Lafayette, IN 47907-2067, USA}
\email{niu6@math.purdue.edu}



\begin{abstract} Let $Y$ be a generic link of a subvariety $X$ of a nonsingular variety $A$. We give a description of the Grauert-Riemenschneider canonical sheaf of $Y$ in terms of the multiplier ideal sheaves associated to $X$ and use it to study the singularities of $Y$. As the first application, we give a criterion when $Y$ has rational singularities and show that log canonical threshold increases and log canonical pairs are preserved in generic linkage. As another application we give a quick and simple liaison method to generalize the results of de Fernex-Ein and Chardin-Ulrich on the Castelnuovo-Mumford regularity bound for a projective variety.

\end{abstract}

\maketitle

\section{Introduction}

\noindent Let $A$ be either a nonsingular affine variety or a projective space $\nP^n$ over the field of complex number $\nC$. Two closed subvarieties of $A$ are said to be geometrically linked if the union of them is a complete intersection in $A$. Let us fix a reduced equidimensional subscheme $X$ of $A$ of codimension $c$ for now. By choosing $c$ equations carefully among the defining equations of $X$ we define a complete intersection $V$ and obtain a subscheme $Y$ as the closure of the complement of $X$ in $V$.  Then $Y$ is geometrically linked to $X$ via $V$, i.e., $Y\cup X=V$ and $Y$ has no common irreducible components with $X$. If such complete intersection $V$ is chosen to be as general as possible, then $Y$ is called a generic link of $X$.

The study of linkage, or the theory of liaison, of algebraic varieties can be traced back to hundreds of years ago. The recent work in this area was initiated by Peskine and Szpiro \cite{PeskineSzpiro:LiaisonI}. After that, linkage has attracted considerable attention and has been developed widely and deeply from both geometric and algebraic point of views.

Let us state the construction of a generic link more concretely to make our introduction clear. Assume further that $X$ is defined by an ideal $I_X$ generated by equations $f_1,\cdots, f_t$. The essential point is to create a general complete intersection $V$ by using the equations $f_i$'s. There are three situations in current research to do this, which we list as follows.

\begin{SitA} Suppose $A=\Spec R$ is affine. A complete intersection $V$ is defined by the equations $\alpha_i=U_{i,1}f_1+\cdots+U_{i,t}f_t$, for $i=1,\cdots, c$, where $U_{i,j}$'s are variables. $V$ is actually defined in an extended space $\Spec R[U_{i,j}]$. By abuse of notation, we still write $A=\Spec R[U_{i,j}]$ and $X$ is defined by $I_X[U_{ij}]$. This only occurs in the introduction. New notation which avoids any ambiguity will be introduced in Definition \ref{def01} and adopted henceforth.
\end{SitA}

\begin{SitB} Suppose $A=\Spec R$ is affine. A complete intersection $V$ is defined by the equations $\alpha_i=a_{i,1}f_1+\cdots+a_{i,t}f_t$, for $i=1,\cdots, c$, where $a_{i,j}$'s are general scalars in $\nC$.
\end{SitB}

\begin{SitC} Suppose $A=\nP^n$ and the homogeneous equations $f_i$'s of $X$ have degrees $d_1\geq d_2\geq \cdots \geq d_t$. A complete intersection $V$ is defined by choosing general equations  $\alpha_i$ from $H^0(\nP^n,I_X(d_i))$, for $i=1,\cdots, c$.
\end{SitC}

\noindent In each of the three situations, once having a complete intersection $V$ in hand, a generic link $Y$ of $X$ can be obtained by algebraic construction, namely, by an ideal $I_Y:=(I_V:I_X)$. (We sometimes use sheaf notation, e.g. $\sI_X$, to view an ideal as a sheaf.)

The terminology of generic linkage usually refers to Situation A. The theory based on this setting has been founded and developed deeply by Huneke and Ulrich in the last three decades, mainly from algebraic side of the story (e.g. \cite{HunekeUlrich:AlgLinkage}, \cite{HunekeUlrich:SturLinkage}). Situation C was studied from the geometric side. The book of Miglior \cite{Migliore:LiaisonTheory} gives a outline of the theory along this direction. It is worth mentioning that the way to construct a generic link in Situation $C$ is quite classical and a typical application can be found in the work of Betram, Ein and Lazarsfeld \cite{BEL} on the Castelnuovo-Mumford regularity bound for a smooth projective variety. Situation B can be considered as a specialization of Situation A as well as a local version of Situation C. As an important technique it has been used in recent studies of singularities, for instance, in the work of de Fernex and Docampo \cite{Roi:JDiscrepancy} and the work of Ein and Mustata \cite{Ein:JetSch}.

However, in contrast to the quick and deep development of singularity theories in the past decades, much less has been known about the singularities in generic linkage. A few special examples have drawn attention recently. One important case, which serves as a guideline of this paper, is a result of Chardin and Ulrich \cite{CU:Reg} that if $X$ is a local complete intersection and has rational singularities then a generic link of $X$ has rational singularities. But the machinery behind  the result seems quite mysterious from the geometric point of view. Meanwhile, experience gained from research in the past gives the intuition that the singularities of a generic link seem to be worse than $X$ itself.

The main theorem of this paper in the context of Situations A, B and C is the following. Recall that $A$ is either a nonsingular affine variety or a projective space and $X$ is a reduced equidimensional subscheme of codimension $c$. We simply write them as a pair $(A,cX)$.

\begin{theorem}\label{mthm}Let $Y$ be a generic link (if it is nonempty) to a pair $(A,cX)$ via $V$. Then $Y$ is reduced equidimensional of codimension $c$, its Grauert-Riemenschneider canonical sheaf is
$$\omega^{GR}_Y\simeq\sI(A,cX)\cdot\sO_Y\otimes \omega_V,$$
where $\sI(A,cX)$ is the multiplier ideal sheaf associated to the pair $(A,cX)$ and $\omega_V$ is the dualizing sheaf of $V$, and it fits into the following commutative diagram
\begin{equation}\label{diagram02}
\xymatrix{
\omega^{GR}_Y \ar[r]^-{\simeq}\ar@{^{(}->}[d]_{\tr} & {\sI(A,cX)\cdot \sO_Y\otimes \omega_V} \ar@{^{(}->}[d]^i   \\
\omega_Y \ar[r]^-{\simeq}  & \sI_X\cdot\sO_Y\otimes \omega_V}
\end{equation}
where the bottom isomorphism is canonical, $\tr$ is the trace map, and the inclusion $i$ is induced by $\sI(A,cX)\hookrightarrow \sI_X$ (cf. Lemma \ref{prop03}). Furthermore, we have inequalities for log canonical thresholds of pairs as $$\lct(A,Y)\geq \lct(A,X).$$
\end{theorem}

It should be noticed that in the application of the theorem the dualizing sheaf $\omega_V$ can be replaced by $\omega_A|_V$ in Situations A and B and by $\omega_{\nP^n}(d_1+\cdots+d_c)|_V$ in Situation C.  Also the generic link $Y$ could be reducible in Situations B and C. There are some special case in which $Y$ is irreducible discussed in Section 4.

The aforementioned result of Chardin and Ulrich is an immediate consequence of the theorem. Also the problem that when a generic link has rational singularities now becomes clear. Furthermore the theorem also implies that log canonical pairs are preserved in generic linkage. It should be mentioned that, as pointed out by the referee, the inequality of log canonical thresholds in the theorem is rather straightforward in Situations B and C.

As another application of the theorem in Situation C of projective spaces, we can generalize and extend the results of de Fernex and Ein
\cite{Ein:VanishLCPairs} and Chardin and Ulrich \cite{CU:Reg} on the  Castelnuovo-Mumford regularity bound for a projective variety. The proof is based on liaison theory and provides a natural, geometric approach to study regularity. The original proofs in \cite{CU:Reg} and \cite{Ein:VanishLCPairs} were very different, and the new proof here finds its way somewhere in the middle, leading to a statement which recovers both results. This was kindly suggested by the referee. The following is a simple version and we will give its full general version in Corollary \ref{cor03}.

\begin{corollary}\label{cor02} Let $X\subset \nP^n$ be a reduced equidimensional subscheme of codimension $c$ defined by the equations of degree $d_1\geq d_2\geq \cdots\geq d_t$ and assume that the pair $(\nP^n,cX)$ is log canonical except possibly at finitely many points. Then
$$\reg X\leq \sum^c_{i=1}d_i-c+1$$
and equality holds if and only if $X$ is a complete intersection in $\nP^n$.
\end{corollary}

We conclude this introduction by briefly stating the organization of this paper. We take Situation A as our framework, following the footprints of Huneke and Ulrich. We prove Theorem \ref{mthm} in Situations A and B in Section 3. Section 4 is devoted to the case of projective varieties and the Castelnuovo-Mumford regularity bound.\\

{\em Acknowledgement}.  Special thanks are due to professor Bernd Ulrich
who introduced the author to this subject and spent his valuable time on discussion. This paper would not be possible without his generous help and encouragement. The author also would like to thank professor Lawrence Ein for his insightful knowledge and inspiring suggestions which enrich the paper. The author's thanks also goes to professor Joseph Lipman for his patient reading and kind suggestions, and goes to the referee for his nice comments and suggestions which improve the quality of the paper.

\section{Generic linkage, singularities and multiplier ideal sheaves}
\noindent Throughout this paper, we work over the field $k:=\nC$. By a variety we mean a reduced irreducible scheme of finite type over $k$. A subscheme is always assumed to be closed. We shall briefly review basic facts about linkage, singularities and multiplier ideal sheaves in this section.

\begin{definition}\label{def00} Let $A$ be either a nonsingular affine variety over $k$ or a projective space $\nP^n_k$. Let $X$ and $Y$ be two subschemes of $A$. We say that $X$ and $Y$ are {\em geometrically linked} if $X\cup Y$ is a complete intersection $V$ in $A$ and $X$ and $Y$ are equidimensional, no embedded components and have no common irreducible components. We also say $Y$ is linked to $X$ via $V$, or vice versa.
\end{definition}

Suppose $Y$ is geometrically linked to $X$ via $V$ in $A$ as defined above. Let $\omega_Y$ be a dualizing sheaf of $Y$ (in this paper, we use dualizing sheaf and canonical sheaf interchangeably). One important fact is an isomorphism
$\omega_Y\simeq\sI_X\cdot\sO_V\otimes \omega_V$,
where $\sI_X$ is the defining ideal sheaf of $X$ in $A$ and $\omega_V$ is a dualizing sheaf of $V$. If $A$ is an affine nonsingular variety, then one can deduce that
\begin{equation}\label{eq00}
\omega_Y\simeq\sI_X\cdot\sO_Y\otimes \omega_A,
\end{equation}
where $\omega_A$ is a dualizing sheaf of $A$. If $A$ is a projective space $\nP^n_k$ and $V$ is cut out by homogeneous equations of degrees $d_1\geq d_2\geq \cdots\geq d_c$, where $c=\codim V$, then one has
\begin{equation}\label{eq04}
\omega_Y\simeq \sI_X\cdot\sO_Y\otimes \omega_A(d_1+\cdots+d_c).
\end{equation}
It is also well-know that $Y$ is Cohen-Macaulay if and only if $X$ is Cohen-Macaulay. The modern approach to study linkage goes back to  Peskine and Szpiro \cite{PeskineSzpiro:LiaisonI}, and we refer to their work for more general theory of linkage.

Definition \ref{def00} suggests that the study of linkage involves an ambient space containing the varieties concerned. Thus it is natural to consider a variety and its ambient space as a pair. Besides, the concept of pairs is a successful approach to study singularities of varieties, which is the main motivation of this paper.

\begin{definition}\label{def02} A {\em pair} $(A,cX)$ contains a nonsingular variety $A$ over $k$, a subscheme $X$ and a nonnegative real number $c$. If $A=\spec R$ is affine and $X$ is defined by an ideal $I_X$, then we say the pair $(A,cX)$ is an affine pair and also use an alternative notation $(R,I^c_X)$.
\end{definition}

Now we give the definition of generic linkage mentioned in Situation A and take it as our framework in Section 3.

\begin{definition}\label{def01} Let $(A_k,cX_k)=(R_k,I^c_{X_k})$ be an affine pair such that $X_k$ is reduced equidimensional and $c=\codim_{A_k}X_k$. We construct a generic link of $X_k$ as follows. Fix a generating set $(f_1,\cdots,f_t)$ of $I_{X_k}$. Let $(U_{ij}), 1\leq i\leq c, 1\leq j\leq t$, be a $c\times t$ matrix of variables. Set $R:=R_k[U_{ij}]$ and $I_X:=I_{X_k}\cdot R_k[U_{ij}]$ and define $A=\spec R$ and $X=\Spec R/I_X$. Notice that $I_X$ is still generated by $(f_1,\cdots,f_t)$ in $R$. We define a complete intersection $V$ inside $A$ by an ideal
$$I_V:=(\alpha_1,\cdots,\alpha_c)=(U_{i,j})\cdot(f_1,\cdots,f_t)^T,$$
that is
$$\alpha_i:=U_{i,1}f_1+U_{i,2}f_2+\cdots+U_{i,t}f_t,\quad\quad\quad\mbox{for } 1\leq i\leq c.$$
Then a generic link $Y$ to $X_k$ via $V$ is defined by an ideal $I_Y:=(I_V:I_X)$.
\end{definition}

\begin{remark} In the definition above, the subscheme $V$ is a complete intersection in $A$ \cite{Hoc73}. The subschemes $Y$ and $X$ of $A$ are  geometrically linked so that $I_Y=(I_V:I_X)$, $I_X=(I_V:I_Y)$ and $I_V=I_X\cap I_Y$ and therefore $Y$ is equidimensional without embedded components and has no common components with $X$ \cite[2.1, 2.5]{HunekeUlrich:DivClass}. Furthermore $Y$ is actually integral \cite[2.6]{HunekeUlrich:DivClass}. In the sequel we only need the fact that $Y$ is reduced equidimensional. It is from (\ref{eq00}) that a dualizing sheaf $\omega_Y$ of $Y$ is $\omega_Y\simeq I_X\cdot \sO_Y\otimes \omega_A$, where $\omega_A$ is a dualizing sheaf of $A$.
\end{remark}

Let $A$ be a nonsingular variety over $k$ and $\{X_i\}$, $i=1,\cdots, m$, are $m$ subschemes of $A$. By Hironaka's resolution of singularities,  there is a birational projective morphism $f:A'\longrightarrow A$ such that $A'$ is a nonsingular variety, $f^{-1}(X_i)=\sum a_{i,j}E_j$ for $i=1,\cdots, m$, where $a_{i,j}$'s are nonnegative integers and $E_j$'s are prime divisors of $A'$ such that the union of $E_j$'s with the exceptional locus $\Exc(f)$ is a simple normal crossing divisor. The morphism $f$ is called a {\em log resolution} of $(A, \sum_i X_i)$.

Let $(A,cX)$ be a pair. Take a log resolution $f:A'\longrightarrow A$ of $(A,X)$ such that $f^{-1}(X)=\sum^s_{i=1} a_iE_i$ and the relative canonical divisor $K_{A'/A}=\sum^s_{i=1} k_iE_i$. We say that the pair $(A,cX)$ is {\em log canonical} if $k_i-c\cdot a_i\geq -1$ for all $i$. The {\em log canonical threshold} of $(A,X)$ is defined to be
$$\lct(A,X):=\min_i\{\frac{k_i+1}{a_i}\}.$$ The {\em multiplier ideal sheaf} $\sI(A,cX)$ associated to the pair $(A,cX)$ is defined to be
$$\sI(A,cX):=f_*\sO_{A'}(K_{A'/A}-\lfloor c\sum a_iE_i\rfloor),$$
where $\lfloor c\sum a_iE_i\rfloor$ is the round down of the $\nQ$-divisor $c\sum a_iE_i$.

\begin{lemma}\label{prop03} Let $A$ be a nonsingular variety over $k$ and $X\subset A$ be a reduced equidimensional subscheme of codimension $c$ defined by an ideal sheaf $\sI_X$. Let $\sI(A,cX)$ be the multiplier ideal sheaf associated to the pair $(A,cX)$. Then $\sI(A,cX)\subseteq \sI_X$.
\end{lemma}
\begin{proof} Since at each generic point $p$ of $X$ one has $\sI(A,cX)_p=\sI_{X,p}$ and $\sI_X$ is radical the result is then clear.
\end{proof}

The following lemma is due to Lawrence Ein, which gives a criterion to compare multiplier ideal sheaves with ideal sheaves.

\begin{Ein}\label{einlemma} Let $A$ be a nonsingular variety and $X\subset A$ be a reduced equidimensional subscheme of codimension $c$ defined by an ideal sheaf $\sI_X$. Then $\sI(A,cX)=\sI_X$ if and only if $\sI(A,(c-1)X)=\sO_A$. In particular if the pair $(A,cX)$ is log canonical then $\sI(A,cX)=\sI_X$.
\end{Ein}
\begin{proof} The inclusion $\sI(A,cX)\subseteq I_X$ is from Lemma \ref{prop03}. It then suffices to show that $\sI_X\subseteq \sI(A,cX)$ if and only if $\sI(A,(c-1)X)$ is trivial. If $c=1$, then there is nothing to prove.  So in the sequel we assume $c>1$. We shall follow notation and terminologies in \cite[Section 7]{Ein:JetSch}.

The inclusion $\sI_X\subseteq\sI(A,cX)$ is true if and only if for any prime divisor $E$ over $X$ we have an inequality
$\ord_E\sI_X+\ord_EK_{\_/A}-c\cdot\ord_E\sI_X\geq 0$.
This is equivalent to the inequality
$\ord_EK_{\_/A}-(c-1)\cdot\ord_E\sI_X\geq 0$,
which is equivalent to $\sI(A,(c-1)X)$ is trivial.

Now suppose $(A,cX)$ is log canonical. This means that for any prime divisor $E$ over $X$ we have $\ord_EK_{\_/A}-c\cdot\ord_E\sI_X\geq -1$. If the center of $E$ is outside $X$ then $\ord_E\sI_X=0$ and $\ord_EK_{\_/A}\geq 0$ since $A$ is nonsingular. And thus  $\ord_EK_{\_/A}-c\cdot\ord_E\sI_X\geq -\ord_E\sI_X$. On the other hand, if the center of $E$ is inside $X$, then $\ord_E\sI_X\geq 1$ and again we have $\ord_EK_{\_/A}-c\cdot\ord_E\sI_X\geq -1\geq -\ord_E\sI_X$. Hence in any case we have $\ord_EK_{\_/A}-c\cdot\ord_E\sI_X\geq -\ord_E\sI_X$, which implies, as showed above, that $\sI(A,(c-1)X)$ is trivial.
\end{proof}

\begin{definition}\label{def03} Let $A$ be a nonsingular variety over $k$ and $X$ a reduced subscheme of $A$. A morphism $\varphi_A:\overline{A}\longrightarrow A$ is a {\em factorizing resolution} of $X$ inside $A$ if the following hold:
\begin{itemize}
\item [(1)] $\varphi_A$ is an isomorphism at the generic point of every irreducible component of $X$. In particular, the strict transform $\overline{X}$ of $X$ is defined.
\item [(2)] The morphism $\varphi_A$ and $\varphi_X:=\varphi_A|_{\overline{X}}$ are resolution of singularities of $A$ and $X$, respectively, and the union of $\overline{X}$ with the exceptional locus $\Exc(\varphi_A)$ has simple normal crossings.
\item [(3)] If $I_X$ and $I_{\overline{X}}$ are the defining ideals of $X$ and $\overline{X}$ in $A$ and $\overline{A}$, respectively, then there exists an effective divisor $G$ on $\overline{A}$ such that
    $$I_X\cdot\sO_{\overline{A}}=I_{\overline{X}}\cdot\sO_{\overline{A}}(-G).$$
    The divisor $G$ is supported on $\Exc(\varphi_A)$ and hence has normal crossing with $\overline{X}$.
\end{itemize}
\end{definition}

\begin{remark}\label{rmk01} The above definition is borrowed from \cite[Definition 2.10]{Ein:MultIdeaMatherDis}. The existence of a factorizing resolution is proved in \cite[Theorem 1.2]{Villamayor:StrenResSing}. In addition, we can assume that the morphism $\varphi_A$ is isomorphic over the open set $A\setminus X$.
\end{remark}

Let us conclude this section by briefly reviewing the definition of rational singularities and Grauert-Riemenschneider canonical sheaves. Let $X$ be a reduced equidimensional scheme of finite type over $k$. Let $f:X'\longrightarrow X$ be a resolution of singularities of $X$. Then the {\em Grauert-Riemenschneider canonical sheaf} $\omega^{GR}_X$ of $X$ is defined to be $\omega^{GR}_X:=f_*\omega_{X'}$, where $\omega_{X'}$ is the canonical sheaf of $X'$. It turns out that the sheaf $\omega^{GR}_X$ is independent on the choice of the resolution of singularities $f$ (cf. \cite{Lazarsfeld:PosAG1}). Furthermore $\omega^{GR}_X$ is canonically a subsheaf of $\omega_X$, a dualizing sheaf of $X$, via a trace map $\tr:\omega^{GR}_X\hookrightarrow \omega_X$. Recall that $X$ has {\em rational singularities} if $f_*{\sO_{X'}}=\sO_X$ and $R^if_*\sO_{X'}=0$ for $i>0$. It is well-know that $X$ has rational singularities if and only if $X$ is Cohen-Macaulay and $\omega^{GR}_X=\omega_X$ (cf. \cite{Kollar:SingOfPairs}).

\section{Generic linkage of affine varieties}

\noindent In this section, we first study the singularities of generic linkages under the framework of Huneke and Ulrich as Definition \ref{def01} in Situation A. In this case, Theorem \ref{mthm} will be proved by using Proposition \ref{prop02}, Proposition \ref{prop04} and Proposition \ref{thm01}. Some direct consequences will also be given which describe singularities of generic linkage. Then we specialize our results to a Zariski open set of a scalar matrices space, which takes care of Situation B.

\begin{proposition}\label{prop02} With notation as in Definition \ref{def01} let $Y$ be a generic link to a pair $(A_k,cX_k)$. Then
$$\omega^{GR}_Y\simeq\sI(A,cX)\cdot\sO_Y\otimes \omega_A,$$
where $\sI(A,cX)$ is the multiplier ideal sheaf associated to the pair $(A,cX)$.
\end{proposition}
\begin{proof} Let $\varphi_k: \overline{A}_k\longrightarrow A_k$ be a factorizing resolution of singularities of $X_k$ inside $A_k$, so that $I_{X_k}\cdot\sO_{\overline{A}_k}=I_{\overline{X}_k}\cdot\sO_{\overline{A}_k}(-G_k)$ where $\overline{X}_k$ is the strict transform of $X_k$, $G_k$ is an effective divisor supported on the exceptional locus of $\varphi_k$, and furthermore $\overline{X}_k$ and the exceptional locus of $\varphi_k$ are simple normal crossings. The morphism $\varphi_k$ can be assumed to be an isomorphism over the open set $A_k\backslash X_k$ (cf. Definition \ref{def03} and Remark \ref{rmk01}).

By tensoring $k[U_{ij}]$ to the factorizing resolution $\varphi_k$, we obtain a factorizing resolution of singularities of $X$ inside $A$ as
$$\varphi:\overline{A}\longrightarrow A,$$
such that $I_X\cdot\sO_{\overline{A}}=I_{ \overline{X}}\cdot\sO_{ \overline{A}}(-G)$, where $\overline{X}$ is the strict transform of $X$, $G$ is an effective divisor supported on the exceptional locus of $\varphi$, and $\overline{X}$ and exceptional locus of $\varphi$ are simple normal crossings. Notice that by the construction, we actually have $\overline{A}=\overline{A}_k\otimes_k \Spec k[U_{ij}]$, $\overline{X}=\overline{X}_k\otimes_k\Spec k[U_{ij}]$ and $G=G_k\otimes_k \Spec k[U_{ij}]$.

\begin{claim}\label{Claim01} The ideal sheaf $I_V\cdot \sO_{\overline{A}}$ has a decomposition as
$$I_V\cdot \sO_{\overline{A}}=I_{\overline{V}}\cdot \sO_{\overline{A}}(-G)$$
where $I_{\overline{V}}$ is an ideal sheaf on $\overline{A}$ and  is a local complete intersection.
\end{claim}

\textit{Proof of Claim \ref{Claim01}.} The question is local. Recall that $\varphi_k: \overline{A}_k\longrightarrow A_k$ is the factorizing resolution of singularities of $X_k$ inside $A_k$. Let $\overline{U}_k=\Spec \overline{R}_k$ be an affine open set of $\overline{A}_k$ such that the effective divisor $G_k$ is defined by an equation $g\in \overline{R}_k$ and we have a decomposition $I_{X_k}\cdot\overline{R}_k=I_{\overline{X}_k}\cdot (g)$ on $\overline{U}_k$. Now since $I_{X_k}\cdot \overline{R}_k=(f_1,\cdots,f_t)\cdot \overline{R}_k$ we can write $f_i=\overline{f}_ig$ where $\overline{f}_i\in \overline{R}_k$ for $i=1,\cdots,t$ so that $I_{\overline{X}_k}=(\overline{f}_1,\cdots,\overline{f}_t)$.

Recall that the factoring resolution $\varphi:\overline{A}\longrightarrow A$ is obtained by tensoring $\Spec k[U_{ij}]$ to the factoring resolution $\varphi_k$. Write $\overline{R}=\overline{R}_k\otimes k[U_{ij}]$ which is a faithfully flat ring extension of $\overline{R}_k$ and then $\overline{U}=\overline{U}_k\otimes \Spec k[U_{ij}]=\Spec \overline{R}$ is an affine open set of $\overline{A}$. Notice that on $\overline{U}$ the ideal $I_{\overline{X}}=I_{\overline{X}_k}\cdot \overline{R}$ and the effective divisor $G$ is still generated by the equation $g$. Recall that the ideal $I_V=(\alpha_1,\cdots,\alpha_c)$, where $\alpha_i=U_{i,1}f_1+U_{i,2}f_2+\cdots+U_{i,t}f_t$. Thus if write $\overline{\alpha}_i=U_{i,1}\overline{f}_1+U_{i,2}\overline{f}_2+\cdots+U_{i,t}\overline{f}_t$ and set $I_{\overline{V}}=(\overline{\alpha}_1,\cdots,\overline{\alpha}_c)$, then $I_{\overline{V}}$ is a complete intersection on $\overline{U}$ and we have a decomposition $I_{V}\cdot \overline{R}=I_{\overline{V}}\cdot (g)$ on $\overline{U}$, which finishes the proof of Claim \ref{Claim01}.\\

Now let $\mu:\widetilde{A}\longrightarrow \overline{A}$ be the blowing-up of $\overline{A}$ along $\overline{X}$ such that $I_{\overline{X}}\cdot \sO_{\widetilde{A}}=\sO_{\widetilde{A}}(-T)$, where $T$ is an exceptional divisor of $\mu$. Denote by $\psi=(\phi\circ\mu): \widetilde{A}\longrightarrow A$. Notice that the supports of divisors $T$,  $\mu^*(G)$ and the exceptional locus of $\psi$ are simple normal crossings. We write $K_{\widetilde{A}/A}$ to be the relative canonical divisor of the morphism $\psi$.

\begin{claim}\label{claim02} We have the following statements.
\begin{enumerate}
\item The ideal sheaf $I_{\overline{V}}\cdot \sO_{\widetilde{A}}$ can be decomposed as
$$I_{\overline{V}}\cdot \sO_{\widetilde{A}}=I_{\widetilde{V}}\cdot \sO_{\widetilde{A}}(-T).$$
where $I_{\widetilde{V}}$ is an ideal sheaf on $\widetilde{A}$ and  defines a local complete intersection $\widetilde{V}$ of codimension $c$.
\item The scheme $\widetilde{V}$ is nonsingular and irreducible and  its dualizing sheaf is $$ \omega_{\widetilde{V}}\simeq\sO_{\widetilde{V}}(K_{\widetilde{A}/A}-c(T+\mu^*G))\otimes\psi^*\omega_A,$$
    where $\omega_A$ is a dualizing sheaf of $A$.
\item The scheme $\widetilde{V}$ is the strict transform of $Y$ via $\psi$.
\end{enumerate}
\end{claim}

\textit{Proof of Claim \ref{claim02}.} We work locally on affine open sets as in the proof of Claim \ref{Claim01}. Assume that $\overline{A}_k=\Spec \overline{R}_k$ and $\overline{A}=\Spec \overline{R}$, where $\overline{R}=\overline{R}_k\otimes_k k[U_{ij}]$. Recall that, as we showed in the proof of Claim \ref{Claim01}, the ideal $I_{\overline{X}}=(\overline{f}_1,\cdots,\overline{f}_t)\cdot \overline{R}$ where for each $i$ the generator $\overline{f}_i$ is in the ring $\overline{R}_k$, and that the complete intersection $I_{\overline{V}}=(\overline{\alpha}_1,\cdots,\overline{\alpha}_c)$, where $\overline{\alpha}_i=U_{i,1}\overline{f}_1+U_{i,2}\overline{f}_2+\cdots+U_{i,t}\overline{f}_t$
for $i=1,\cdots, c$. Now take a canonical affine cover of $\widetilde{A}$, say $U=\Spec \overline{R}[\overline{f}_2/\overline{f}_1,\cdots, \overline{f}_t/\overline{f}_1]$ such that the exceptional divisor $T$ is given by the element $\overline{f}_1$ on $U$. For $i=1,\cdots,c$ write $\widetilde{\alpha}_i=U_{i,1}+U_{i,2}\overline{f}_2/\overline{f}_1+\cdots+U_{i,t}\overline{f}_t/\overline{f}_1$
and set $I_{\widetilde{V}}=(\widetilde{\alpha}_1,\cdots,\widetilde{\alpha}_c)$. Then on the open set $U$ we have $I_{\overline{V}}\cdot \sO_U=I_{\widetilde{V}}\cdot (f_1)$ and $I_{\widetilde{V}}$ defines an irreducible nonsingular variety of $\widetilde{V}$ on $U$, which prove the statement (1) and the first part of the statement (2) in the claim.

Next we compute the dualizing sheaf of $\widetilde{V}$. Notice that we have $I_V\cdot\sO_{\widetilde{A}}=I_{\widetilde{V}}\cdot \sO_{\widetilde{A}}(-T-\mu^*G)$. Since $I_V$ is a complete intersection in $A$ of codimension $c$ we have a surjective morphism
$\oplus^c\sO_A\longrightarrow I_V\longrightarrow 0$,
which induces on $\widetilde{A}$ a surjective morphism
$$\bigoplus^c\sO_{\widetilde{A}}(T+\mu^*G)\longrightarrow I_{\widetilde{V}}\longrightarrow 0.$$
Thus it is clear that the determinant of the normal bundle of $\widetilde{V}$ inside $\widetilde{A}$ is $$\det N_{\widetilde{V}/\widetilde{A}}=\sO_{\widetilde{V}}(-c(T+\mu^*G)).$$ Then by the adjunction formula, we have $$\omega_{\widetilde{V}}\simeq\omega_{\widetilde{A}}\otimes \det N_{\widetilde{V}/\widetilde{A}}=\sO_{\widetilde{V}}(K_{\widetilde{A}/A}-c(T+\mu^*G))\otimes \psi^*\omega_A,$$
which proves the second part of the statement (2) in the claim. For the statement (3), just notice that the morphism $\psi$ is an isomorphism at the generic point of $Y$. Thus we finish the proof of Claim \ref{claim02}.\\

Now twisting the short exact sequence $0\longrightarrow I_{\widetilde{V}}\longrightarrow \sO_{\widetilde{A}}\longrightarrow \sO_{\widetilde{V}}\longrightarrow 0$ by the divisor $\sO_{\widetilde{A}}(K_{\widetilde{A}/A}-c(T+\mu^*G))\otimes \psi^*\omega_A$,
we obtain an exact sequence
\begin{equation}\label{eq01}
0\rightarrow I_{\widetilde{V}}\cdot \sO_{\widetilde{A}}(K_{\widetilde{A}/A}-c(T+\mu^*G))\otimes \psi^*\omega_A\rightarrow \sO_{\widetilde{A}}(K_{\widetilde{A}/A}-c(T+\mu^*G))\otimes \psi^*\omega_A\rightarrow \omega_{\widetilde{V}}\rightarrow 0.
\end{equation}
Push down this sequence via $\psi$.  Notice that by the definition of multiplier ideal sheaves, we obtain
$$\psi_*\sO_{\widetilde{A}}(K_{\widetilde{A}/A}-c(T+\mu^*G))=\sI(A,cX).$$
Now we make the following claim.

\begin{claim}\label{claim03} We have the following statements for the sequence \ref{eq01}.
\begin{enumerate}
\item Let $\sI(A,cV)$ be the multiplier ideal sheaf associated to the pair $(A,cV)$, then
$$\psi_*(I_{\widetilde{V}}\cdot \sO_{\widetilde{A}}(K_{\widetilde{A}/A}-c(T+\mu^*G)))=\sI(A,cV).$$
\item  We have the vanishing
$$R^i\psi_*(I_{\widetilde{V}}\cdot \sO_{\widetilde{A}}(K_{\widetilde{A}/A}-c(T+\mu^*G))=0,\quad \mbox{for }i>0.$$
\end{enumerate}
\end{claim}

\textit{Proof of Claim \ref{claim03}.} Let $\nu: A'\longrightarrow \widetilde{A}$ be the blowing-up of $\widetilde{A}$ along $\widetilde{V}$ such that $I_{\widetilde{V}}\cdot \sO_{A'}=\sO_{A'}(-S)$, where $S$ is an exceptional divisor of $\nu$. Notice that $K_{A'/\widetilde{A}}=(c-1)S$ and $K_{A'/A}=K_{A'/\widetilde{A}}+\nu^*K_{\widetilde{A}/A}$. Thus we have
$$-S+\nu^*K_{\widetilde{A}/A}-c\nu^*(T+\mu^*G)=K_{A'/A}-c(S+\nu^*T+(\mu\circ\nu)^*G).$$
Write the divisor $D:=(S+\nu^*T+(\mu\circ\nu)^*G)$. We notice that $I_V\cdot\sO_{A'}=\sO_{A'}(-D)$. Since $\nu_*\sO_{A'}(-S)=I_{\widetilde{V}}$ and $R^i\nu_*\sO_{A'}(-S)=0$ for $i>0$, we obtain that $$\nu_*\sO_{A'}(K_{A'/A}-cD)=I_{\widetilde{V}}\cdot\sO_{\widetilde{A}}(K_{\widetilde{A}/A}-c(T+\mu*G))$$ and
$$R^i\nu_*\sO_{A'}(K_{A'/A}-cD)=0,\quad\mbox{for }i>0.$$
Write $f:=\psi\circ\nu:A'\longrightarrow A$. The divisor $-D$ is $f$-nef and then by the Kawamata-Viehweg vanishing theorem we have
$$R^if_*\sO_{A'}(K_{A'/A}-cD)=0,\quad\mbox{for }i>0.$$
Also by the definition of multiplier ideal sheaves we have $f_*\sO_{A'}(K_{A'/A}-cD)=\sI(A,cV)$.
Now by using the spectral sequence
$$E^{p,q}_2=R^p\psi_*(R^q\nu_*\sO_{A'}(K_{A'/A}-cD))\Rightarrow R^{p+q}f_*\sO_{A'}(K_{A'/A}-cD),$$
we immediately have that
$$R^i\psi_*(I_{\widetilde{V}}\cdot \sO_{\widetilde{A}}(K_{\widetilde{A}/A}-c(T+\mu^*G))=0,\quad \mbox{for }i>0,$$
and
$$\psi_*(I_{\widetilde{V}}\cdot \sO_{\widetilde{A}}(K_{\widetilde{A}/A}-c(T+\mu^*G)))=\sI(A,cV),$$
which finish the proof of Claim \ref{claim03}.
\\

Finally, we push down the sequence (\ref{eq01}) via $\psi$ to obtain  an exact sequence
$$0\longrightarrow \sI(A,cV)\otimes \omega_A\longrightarrow \sI(A,cX)\otimes \omega_A\longrightarrow \omega^{GR}_Y\longrightarrow 0.$$
Thus by restricting to $Y$ we see $\sI(A,cX)\cdot \sO_Y\otimes \omega_A\simeq\omega^{GR}_Y$, which proves the proposition.

\end{proof}

Next we show that the isomorphism proved in Proposition \ref{prop02} is canonical in the sense that it fits into the  commutative diagram \ref{diagram02}. To this end, we need to track carefully the isomorphisms constructed in the proposition. Since canonical sheaves are only unique up to isomorphism we need to fix our canonical sheaves uniformly in arguments. There is a canonical canonical sheaf, namely sheaf of regular differential forms, developed by Kunz \cite{Kunz:RegDiffForm}, offering a concrete way to do so. We follow the notation of Lipman \cite{Lipman:DuaDiffRes} to denote by $\widetilde{\omega}$ the sheaf of regular differential forms.

Let $Z$ be a reduced of pure dimension $d$ scheme of finite type over $k$. Denote by $\sK_Z$ the locally constant sheaf of total quotient ring of $\sO_Z$. Its ring of global sections is $K(Z):=\sK_Z(Z)$ which is a product of residue fields of the generic points of $Z$. Let $\Omega^d_{Z/k}:=\wedge^d\Omega^1_{Z/k}$ be the $d$-th exterior power of the sheaf of K$\ddot{a}$hlar differential one-form. Let $\bar{\Omega}^d_{K(Z)/k}$ be the locally constant sheaf of meromorphic $d$-forms on $Z$ so that its module of global sections is $\bar{\Omega}^d_{K(Z)/k}(Z)=\Omega^d_{K(Z)/k}$. The sheaf $\widetilde{\omega}_Z$ of regular differential forms of degree $d$ of $Z$ is defined in \cite[Section 3]{Kunz:RegDiffForm} and it is a subsheaf of $\bar{\Omega}^d_{Z/k}$. Now let $f:Z'\longrightarrow Z$ be a resolution of singularities of $Z$ so that $f$ is isomorphic at the generic points of $Z$. Then pushdown the inclusion $\widetilde{\omega}_{Z'}\hookrightarrow \bar{\Omega}^d_{K(Z')/k}$ via $f$ we have $f_*\widetilde{\omega}_{Z'}\hookrightarrow f_*\bar{\Omega}^d_{K(Z')/k}$. But since $f$ is generically isomorphism we see that $K(Z')=K(Z)$ and $f_*\bar{\Omega}^d_{K(Z')/k}=\bar{\Omega}^d_{K(Z)/k}$. Thus $f_*\widetilde{\omega}_{Z'}$ is naturally included in $\bar{\Omega}^d_{K(Z)/k}$ as a subsheaf. The trace map $\tr: f_*\widetilde{\omega}_{Z'}\hookrightarrow \widetilde{\omega}_{Z}$ is then the natural inclusion as subsheaves of $\bar{\Omega}^d_{K(Z)/k}$.

\begin{proposition}\label{prop04} With notation as in Proposition \ref{prop02}, the isomorphism $\omega^{GR}_Y\simeq\sI(A,cX)\cdot\sO_Y\otimes \omega_A$ is canonical in the sense that it fits into the following commutative diagram
\begin{equation}\label{diagram01}
\xymatrix{
\omega^{GR}_Y \ar[r]^-{\simeq}\ar@{^{(}->}[d]_{\tr} & {\sI(A,cX)\cdot \sO_Y\otimes \omega_A} \ar@{^{(}->}[d]^i   \\
\omega_Y \ar[r]^-{\simeq}  & I_X\cdot\sO_Y\otimes \omega_A
}\end{equation}
where the bottom isomorphism is given by \ref{eq00}, $\tr$ is the trace map, and the inclusion $i$ is induced by $\sI(A,cX)\hookrightarrow I_X$ (cf. Lemma \ref{prop03}).
\end{proposition}

\begin{proof}
Keep notation and construction as in the proof of Proposition \ref{prop02}. Assume that $d:=\dim X$. We first make the adjunction isomorphism $\omega_V\simeq \omega_A\otimes_{\sO_A}\sO_V$ precisely by using regular differential forms mentioned above. Recall that $V$ is a complete intersection defined by $I_V=(\alpha_1,\cdots, \alpha_c)$. Then it is clear that $I_V/I^2_V=\oplus\sO_V\bar{\alpha}_i$. Following notation of \cite[Section 13]{Lipman:DuaDiffRes} we set a sheaf
$$\sH_{V,A}:=\sHom_{\sO_V}(\wedge^c I_V/I^2_V,\widetilde{\omega}_A/I_V\widetilde{\omega}_A),$$
which is torsion free since $\widetilde{\omega}_A$ is locally free. Notice that $\sH_{V,A}=\det N_{V/A}\otimes_{\sO_A} \widetilde{\omega}_A$ where $N_{V/A}$ is the normal sheaf of $V$ in $A$.  We have the following commutative diagram
\begin{equation}\label{eq10}
\xymatrix{
\widetilde{\omega}_V \ar@{^{(}->}[d]\ar[r]^-{\simeq} & \sH_{V,A} \ar[d]^{i_V} \ar[r]^-{\simeq} &  \widetilde{\omega}_A \otimes_{\sO_A}\sO_V \ar[d]^{j_V}\\
\bar{\Omega}^d_{K(V)/k} \ar[r]^-{\simeq}_-{a_V}  & \sH_{V,A}\otimes_{\sO_V}\sK_V \ar[r]^-{\simeq}_-{b_V}& \widetilde{\omega}_A \otimes_{\sO_A} \sK_V
}\end{equation}
where the left commutative square follows from \cite[Theorem 2.3]{Hubl:AdjMorRegDiffForms} (see also \cite[Corollary 13.7]{Lipman:DuaDiffRes}) and the right commutative square is a consequence of that the sheaves inside $\sHom$ are all locally free and $\det N_{V/A}\simeq\sO_V$. The morphism $i_V$ and $j_V$ are injective because $\sH_{V,A}$ and $\widetilde{\omega}_A \otimes_{\sO_A}\sO_V$ are torsion free. Furthermore
the vertical morphisms in the diagram can be thought of induced by tensoring with the natural inclusion $\sO_V\hookrightarrow \sK_V$. Thus $\sH_{V,A}$ and $\widetilde{\omega}_A \otimes_{\sO_A}\sO_V$ are naturally as subsheaves of those locally constant sheaves at the bottom of the diagram. The adjunction isomorphism $\widetilde{\omega}_V\simeq \widetilde{\omega}_A\otimes_{\sO_A}\sO_V$ is then induced by $a^{-1}_V\circ b^{-1}_V$, i.e. $\widetilde{\omega}_V=(a^{-1}_V\circ b^{-1}_V)(\widetilde{\omega}_A\otimes_{\sO_A} \sO_V)$. The isomorphisms $a_V$ and $b_V$ can be described precisely and determined completely at each generic point of $V$. Assume $v$ is one generic point of $V$ with the residue field $k(v)$. The local ring $\sO_{A,v}$ has a regular system of parameters $\alpha_1,\cdots, \alpha_c,x_1,\cdots, x_d$. Then locally at an open set of $V$ containing only $v$ the sheaf $\bar{\Omega}^d_{K(V)/k}=k(v)d\bar x_1\wedge\cdots\wedge d\bar{x}_d$, the sheaf $\sH_{V,A}\otimes \sK_V=k(v)\xi$ where $\xi$ maps $\bar{\alpha}_1\wedge\cdots\wedge\bar{\alpha}_c$ to $d\alpha_1\wedge\cdots d\alpha_c\wedge dx_1\wedge\cdots\wedge d{x_d}$ and the sheaf $\widetilde{\omega}_A \otimes_{\sO_A} \sK_V=k(v)d\alpha_1\wedge\cdots d\alpha_c\wedge dx_1\wedge\cdots\wedge d{x_d}$. Then on this open neighborhood of $v$ the bottom line of the diagram \ref{eq10} can be written as
$$k(v)d\bar x_1\wedge\cdots\wedge d\bar{x}_d\longrightarrow k(v)\xi\longrightarrow k(v) d\alpha_1\wedge\cdots\wedge d\alpha_c\wedge dx_1\cdots\wedge dx_d.$$
The isomorphisms $a_V$ and $b_V$ are defined in the way that $a_V(d\bar x_1\wedge\cdots\wedge d\bar{x}_d)=\xi$ and $b_V(\xi)=d\alpha_1\wedge\cdots\wedge d\alpha_c\wedge dx_1\cdots\wedge dx_d$.

Next we need to make the isomorphism $\omega_Y\simeq I_X\cdot \sO_Y\otimes \omega_A$ clearly. Write the sheaf $$\sH_{Y,V}:=\sHom_{\sO_Y}(\sO_Y,\widetilde{\omega}_V/I_{Y/V}\widetilde{\omega}_V)$$ where $I_{Y/V}=I_Y\cdot \sO_V$. It is a torsion free since $\widetilde{\omega}_V$ is locally free. There is a fundamental local homomorphism (cf. \cite[13.1]{Lipman:DuaDiffRes})
$$h:\sHom_{\sO_V}(\sO_Y,\widetilde{\omega}_V)\longrightarrow \sH_{Y,V}$$
which in our case is induced by the natural quotient morphism $\widetilde{\omega}_V\longrightarrow \widetilde{\omega}_V/I_{Y/V}\widetilde{\omega}_V$.
We have the following commutative diagram
\begin{equation}\label{eq14}
\xymatrix{
\widetilde{\omega}_Y \ar@{^{(}->}[dd]\ar[r]^-{\simeq} & \sHom_{\sO_V}(\sO_Y,\widetilde{\omega}_V) \ar[d]^h \ar[r]&  \widetilde{\omega}_V\otimes_{\sO_V}I_X\cdot \sO_V \ar[d]^{u_Y}\\
 & \sH_{Y,V} \ar[d]^{i_Y} \ar[r]^-{\simeq} &  \widetilde{\omega}_V \otimes_{\sO_V}\sO_Y \ar[d]^{j_Y}\\
\bar{\Omega}^d_{K(Y)/k} \ar[r]^-{\simeq}_-{a_Y}  & \sH_{Y,V}\otimes_{\sO_Y}\sK_Y \ar[r]^-{\simeq}_-{b_Y}& \widetilde{\omega}_V \otimes_{\sO_V} \sK_Y.
}\end{equation}
The left hand side big square is commutative checked by definition directly. The right two small squares are commutative because  $\widetilde{\omega}_V$ is locally free and  $\sHom_{\sO_V}(\sO_Y,\sO_V)=I_X\cdot \sO_V$. The morphisms $i_Y$ and $i_Y\circ h$ are injective since sheaves involved are all torsion free. Now let $v$ be a generic point of $Y$ which is also a generic point of $V$. Suppose that the local ring $\sO_{A,v}$ has a regular system of parameters $\alpha_1,\cdots, \alpha_c,x_1,\cdots, x_d$. Then locally at an open set of $Y$ containing only $v$ we see that the sheaves $\bar{\Omega}^d_{K(Y)/k}=k(v)d\bar x_1\wedge\cdots\wedge d\bar{x}_d$,  $\sH_{Y,V}\otimes_{\sO_Y}\sK_Y=k(v)\xi_Y$ where $\xi_Y$ maps $1$ to $d\bar x_1\wedge\cdots\wedge d\bar{x}_d$,  and $\widetilde{\omega}_V\otimes_{\sO_V}\sK_Y=k(v)d\bar x_1\wedge\cdots\wedge d\bar{x}_d$. Thus it is easy to check that the morphism $b_Y\circ a_Y$ is the identity. Hence in the locally constant sheaf $\bar{\Omega}^d_{K(Y)/k}=\widetilde{\omega}_V\otimes_{\sO_Y}\sK_Y$ the canonical sheaf $\widetilde{\omega}_Y$ is exactly the sheaf $\widetilde{\omega}_V\otimes_{\sO_V}I_X\cdot \sO_V$. Furthermore since $Y$ is generic linked to $X$ via $V$ we see that $I_X\cdot \sO_V=I_X\cdot \sO_Y$.

Now tensoring the diagram (\ref{eq10}) with $\sO_Y$ over $\sO_V$. Notice that $\sK_V\otimes_{\sO_V}\sO_Y=\sK_Y$ since $I_{Y/V}\cdot\sK_V=\sK_X$. And combining the diagram (\ref{eq14}) together we have
\begin{equation}
\xymatrixcolsep{5pc}\xymatrix{
\widetilde{\omega}_Y \ar@{^{(}->}[dd]\ar[r]^-{=} & \widetilde{\omega}_V\otimes_{\sO_V}I_X\cdot \sO_V \ar[d]^{u_Y}\ar[r] & \widetilde{\omega}_A\otimes_{\sO_A} I_X\cdot \sO_V \ar[d]\\
 &  \widetilde{\omega}_V \otimes_{\sO_V}\sO_Y \ar[d]^{j_Y} \ar[r]& \widetilde{\omega}_A\otimes_{\sO_A}\sO_Y \ar[d]\\
\bar{\Omega}^d_{K(Y)/k} \ar[r]^-{=}_-{b_Y\circ a_Y} & \widetilde{\omega}_V \otimes_{\sO_V} \sK_Y \ar[r]^-{\simeq}_-{(b_V\circ a_V)\otimes 1_{Y}}& \widetilde{\omega}_A\otimes_{\sO_A} \sK_V\otimes_{\sO_V}\sO_Y.
}\end{equation}
The two horizontal morphisms on the top and right of the diagram can be thought of as the restriction of the right bottom morphism of locally constant sheaves on their subsheaves.
Thus the sheaf $\widetilde{\omega}_Y$ is the image of $I_X\otimes \widetilde{\omega}_A$ under the following morphisms
\begin{equation}\label{eq12}
I_X\otimes \widetilde{\omega}_A\hookrightarrow \widetilde{\omega}_A\longrightarrow \widetilde{\omega}_A\otimes \sO_Y\longrightarrow \widetilde{\omega}_A\otimes_{\sO_A} \sK_Y \xrightarrow{(b_V\circ a_V)\otimes 1_{Y}^{-1}} \bar{\Omega}^d_{K(Y)/k}
\end{equation}
where except $(b_V\circ a_V)\otimes 1_{Y}^{-1}$ the rest morphisms are all natural ones.

Now we look at the canonical sheaf $\widetilde{\omega}_{\widetilde{V}}$. Denote by $E:=T+\mu^*G$ an effective divisor on $\widetilde{A}$. Recall that $\widetilde{V}$ is locally a complete intersection on $\widetilde{A}$ and $\wedge^cI_{\widetilde{V}}/I^2_{\widetilde{V}}=\sO_{\widetilde{A}}(cE)$. The normal sheaf $N_{\widetilde{V}/\widetilde{A}}=\sO_{\widetilde{V}}(-cE)$. We define the sheaf
$$\sH_{\widetilde{V},\widetilde{A}}=\sHom_{\sO_{\widetilde{V}}}(\wedge^cI_{\widetilde{V}}/I^2_{\widetilde{V}},\widetilde{\omega}_{\widetilde{A}}/
I_{\widetilde{V}}\widetilde{\omega}_{\widetilde{A}})$$
which is torsion free since $\widetilde{\omega}_{\widetilde{A}}$ is free. Exactly as in the situation of $V$ on $A$, we have the following commutative diagram
\begin{equation}
\xymatrix{
\widetilde{\omega}_{\widetilde{V}} \ar@{^{(}->}[d]\ar[r]^-{\simeq} & \sH_{\widetilde{V},\widetilde{A}} \ar[d] \ar[r]^-{\simeq} &  \widetilde{\omega}_{\widetilde{A}} \otimes_{\sO_{\widetilde{A}}}\sO_{\widetilde{V}}(-cE) \ar[d]\\
\bar{\Omega}^d_{K(\widetilde{V})/k} \ar[r]^-{\simeq}_-{a_{\widetilde{V}}}  & \sH_{\widetilde{V},\widetilde{A}}\otimes_{\sO_{\widetilde{V}}}\sK_{\widetilde{V}} \ar[r]^-{\simeq}_-{b_{\widetilde{V}}}& \widetilde{\omega}_A \otimes_{\sO_A} \sK_V
}
\end{equation}
Thus the sheaf $\widetilde{\omega}_{\widetilde{V}}$ is the image of the sheaf $\widetilde{\omega}_{\widetilde{A}}(-cE)$ under the morphisms
\begin{equation}\label{eq11}
\widetilde{\omega}_{\widetilde{A}}(-cE) \hookrightarrow\widetilde{\omega}_{\widetilde{A}}\longrightarrow \widetilde{\omega}_{\widetilde{A}}\otimes \sO_{\widetilde{V}}\longrightarrow \widetilde{\omega}_{\widetilde{A}}\otimes_{\sO_{\widetilde{A}}}
\sK_{\widetilde{V}}\xrightarrow{(a_{\widetilde{V}}\circ b_{\widetilde{V}})^{-1}}\bar{\Omega}^d_{K(\widetilde{V})/k}
\end{equation}
where except of $(a_{\widetilde{V}}\circ b_{\widetilde{V}})^{-1}$ all morphisms are natural ones.
Push down (\ref{eq11}) via $\psi$. Notice that $\psi_*(\widetilde{\omega}_{\widetilde{A}}(-cE))=\sI(A,cX)\otimes \widetilde{\omega}_A$ and $\psi_*\widetilde{\omega}_{\widetilde{A}}=\widetilde{\omega}_A$. Also since the birational morphism $\psi$ is an isomorphism over $A\backslash X$ so it is an isomorphism around generic points of $Y$ and $\widetilde{V}$ and therefore $\psi_*(\widetilde{\omega}_{\widetilde{A}}\otimes_{\sO_{\widetilde{A}}}
\sK_{\widetilde{V}})=\widetilde{\omega}_A\otimes_{\sO_A} \sK_Y $ and $\psi_*\bar{\Omega}^d_{K(\widetilde{V})/k}=\bar{\Omega}^d_{K(Y)/k}$ as locally constant sheaves. Furthermore at one generic point $v$ of $Y$ which is identical to a generic point of $\widetilde{V}$ since $\psi$ is an isomorphism around $v$, we can choose the same local equation of the local ring $\sO_{A,v}$ and $\sO_{\widetilde{A},v}$, for instance, $\alpha_1,\cdots, \alpha_c,x_1,\cdots, x_d$. Then we can check that the morphism $(a_{\widetilde{V}}\circ b_{\widetilde{V}})^{-1}$ is the same as the morphism $(b_V\circ a_V)\otimes 1_{Y}^{-1}$. Thus we see that the sheaf $\psi_*\widetilde{\omega}_{\widetilde{V}}$ is the image of $\sI(A,cX)\otimes \omega_A$ under the morphisms
\begin{equation}\label{eq13}
\sI(A,cX)\otimes\widetilde{\omega}_A \hookrightarrow\widetilde{\omega}_A\longrightarrow \psi_*(\widetilde{\omega}_{\widetilde{A}}\otimes \sO_{\widetilde{V}})\longrightarrow \widetilde{\omega}_A\otimes_{\sO_A} \sK_Y \xrightarrow{(b_V\circ a_V)\otimes 1_{Y}^{-1}}\bar{\Omega}^d_{K(Y)/k.}
\end{equation}
The fact that $\sI(A,cX)\otimes \widetilde{\omega}_A$ is mapped surjectively to $\psi_*\widetilde{\omega}_{\widetilde{V}}$ is guaranteed by Claim \ref{claim03} in the proof of Proposition \ref{prop02}.

Finally we compare (\ref{eq12}) with (\ref{eq13}). We have the following commutative diagram on $\widetilde{A}$
\begin{equation}
\xymatrix{
\psi^*\widetilde{\omega}_A \ar[d]\ar[r] & \psi^*\widetilde{\omega}_A\otimes \sO_{\widetilde{V}} \ar[d] \ar[r] &  \psi^*\widetilde{\omega}_A\otimes_{ \sO_{\widetilde{A}}} \sK_{\widetilde{V}} \ar@{=}[d]\\
\widetilde{\omega}_{\widetilde{A}} \ar[r] & \widetilde{\omega}_{\widetilde{A}}\otimes\sO_{\widetilde{V}} \ar[r]& \widetilde{\omega}_{\widetilde{A}} \otimes_{\sO_{\widetilde{A}}} \sK_{\widetilde{V}}
}
\end{equation}
where the vertical morphisms are induced by the morphism $\psi^*\Omega^1_{A/k}\longrightarrow \Omega^1_{\widetilde{A}/k}$. Push down the diagram and notice that $\psi_*\widetilde{\omega}_{\widetilde{A}}=\widetilde{\omega}_A$ then we have the commutative diagram
\begin{equation}
\xymatrix{
& \widetilde{\omega}_A\otimes \psi_*\sO_{\widetilde{V}} \ar[d] \ar[r] &  \widetilde{\omega}_A\otimes_{\sO_A} \sK_Y \ar@{=}[d]\\
\widetilde{\omega}_A \ar[ur]^g\ar[r] & \psi_*(\widetilde{\omega}_{\widetilde{A}}\otimes\sO_{\widetilde{V}}) \ar[r]& \widetilde{\omega}_A\otimes_{\sO_A} \sK_Y
}
\end{equation}
Since $\widetilde{\omega}_A\otimes \psi_*\sO_{\widetilde{V}}$ is naturally a $\sO_Y$-module the morphism $g$ then factors through  $\widetilde{\omega}_A\longrightarrow \widetilde{\omega}_A\otimes\sO_Y\longrightarrow\widetilde{\omega}_A\otimes\psi_*\sO_{\widetilde{V}} $. Now it is clear that the proposition follows from (\ref{eq12}) and (\ref{eq13}).
\end{proof}

\begin{corollary}\label{cor01} With notation as in Definition \ref{def01}, let $Y$ be a generic link to an affine pair $(A_k,cX_k)$. Then $\omega^{GR}_Y=\omega_Y$ if and only if $I_{X_k}=\sI(A_k,cX_k)$,
where $\sI(A_k,cX_k)$ is the multiplier ideal sheaf associated to the pair $(A_k,cX_k)$.
\end{corollary}
\begin{proof} Let $\sI(A,cX)$ be the multiplier ideal sheaf associated to the pair $(A,cX)$. By the commutative diagram (\ref{diagram01}) we have $\omega^{GR}_Y=\omega_Y$ if and only if $\sI(A,cX)\cdot\sO_Y=I_X\cdot\sO_Y$ if and only if $I_X+I_Y=\sI(A,cX) +I_Y$. Intersecting with $I_X$  and noticing that $I_X\cap I_Y=I_V$ and $\sI(A,cX)\subseteq I_X$, we see that  $I_X+I_Y=\sI(A,cX) +I_Y$ if and only if $I_X=\sI(A,cX) +I_V$.

Now since the morphism $A\longrightarrow A_k$ is smooth, we then have $\sI(A,cX)=\sI(A_k,cX_k)\cdot\sO_A$ by \cite[9.5.45]{Lazarsfeld:PosAG2}. Also notice that $I_X=I_{X_k}\cdot \sO_A$ and the ring extension $\sO_{A_k}\longrightarrow \sO_A$ is faithfully flat. Thus intersecting with $\sO_{A_k}$, we conclude that $I_X=\sI(A,cX) +I_V$ if and only if $I_{X_k}=\sI(A_k,cX_k)$.
\end{proof}

Now we can easily deduce a criterion when a generic link has rational singularities. It turns out that multiplier ideal sheaves determine rational singularities of a generic link.

\begin{corollary}\label{prop06} With notation as in Definition \ref{def01}, let $Y$ be a generic link to an affine pair $(A_k,cX_k)$. Then $Y$ has rational singularities if and only if $X_k$ is Cohen-Macaulay and $I_{X_k}=\sI(A_k,cX_k)$,
where $\sI(A_k,cX_k)$ is the multiplier ideal sheaf associated to the pair $(A_k,cX_k)$.
\end{corollary}
\begin{proof} $Y$ has rational singularities if and only if $Y$ is Cohen-Macaulay and $\omega^{GR}_Y=\omega_Y$. Then the result is clear from above.
\end{proof}

\begin{corollary}\label{prop07} With notation as in Definition \ref{def01}, let $Y$ be a generic link to an affine pair $(A_k,cX_k)$. Suppose that the pair $(A_k,cX_k)$ is log canonical and $X_k$ is Cohen-Macaulay. Then $Y$ has rational singularities.
\end{corollary}
\begin{proof} By Ein's Lemma, that $(A_k,cX_k)$ is log canonical implies $I_{X_k}=\sI(A_k,cX_k)$. Then the result follows from above.
\end{proof}

\begin{remark} Let us go back to the result of Chardin and Ulrich mentioned in Introduction. Still with notation as in Definition \ref{def01} suppose that $X_k$ is a local complete intersection with rational singularities. Then by the inversion of adjunction \cite{Ein:JetSch} the pair $(A_k,cX_k)$ is log canonical. Thus Corollary \ref{prop06} says that a generic link $Y$ of $X_k$ has rational singularities.
\end{remark}


\begin{proposition}\label{thm01} With notation as in Definition \ref{def01} let $(A_k,cX_k)$ be an affine pair and let $Y$ be a generic link to $X$ via a complete intersection $V$. Then
$$\lct(A,Y)\geq \lct(A,V)=\lct(A,X)=\lct(A_k,X_k).$$
\end{proposition}
\begin{proof} Keep notation and construction as in the proof of Proposition \ref{prop02}. First of all it is clear that $\lct(A,X)=\lct(A_k,X_k)$ and since $I_V\subseteq I_Y$ we have $\lct(A,Y)\geq \lct(A,V)$. Thus it suffices to show
$\lct(A,V)=\lct(A,X)$.

Recall that $\varphi_k:\overline{A}_k\longrightarrow A_k$ is a factorization resolution of singularities of $X_k$ inside $A_k$ and $\overline{X}_k$ is the strict transform of $X_k$. Denote by $\Exc(\varphi_k)=\cup^s_{i=1}E^k_{i}$ the exceptional locus of $\varphi_k$ where $E^k_i$ are prime divisors with normal crossing support. Then $\overline{X}_k$ has normal crossing with $E^k_1,\cdots,E^k_s$ (cf. Definition \ref{def03}). We then can write the effective divisor $G_k=\sum^s_{i=1}a_iE^k_i$ and the relative canonical divisor $K_{\overline{A}_k/A_k}=\sum^s_{i=1}k_iE^k_i$

Recall also that the factorizing resolution of singularities $\varphi:\overline{A}\longrightarrow A$ of $X$ inside $A$ is then obtained by tensoring $\Spec k[U_{ij}]$ with the resolution $\varphi_k$. Write $E_i:=E^k_i\otimes_k\Spec k[U_{ij}]$ for $i=1,\cdots, s$. Then it is clear that the exceptional locus of $\varphi$ is $\Exc(\varphi)=\cup^s_{i=1}E_i$ with simple normal crossing support, and the strict transform $\overline{X}$ of $X$ has normal crossing with $E_1,\cdots, E_s$. Furthermore the effective divisor $G=\sum^s_{i=1}a_iE_i$ and the relative canonical divisor $K_{\overline{A}/A}=\sum^s_{i=1}k_iE_i$.

Now the morphism $\psi:\widetilde{A}\longrightarrow A$ is the composition of $\varphi$ with the blowing-up $\mu:\widetilde{A}\longrightarrow A$ along $\overline{X}$ with the exceptional divisor $T$. We set $\widetilde{E}_i=\mu^*(E_i)$ for $i=1,\cdots, s$. Since $\overline{X}$ has normal crossing with $E_i$ we see that $\widetilde{E}_i$ is a prime divisor and $\Exc(\psi)=T\cup^s_{i=1}\widetilde{E}_i$
has simple normal crossing support. Thus $\psi:\widetilde{A}\longrightarrow A$ is a log resolution of $X$ inside $A$. Notice that $K_{\widetilde{A}/\overline{A}}=(c-1)T$. We then can write
\begin{equation}\label{eq15}
K_{\widetilde{A}/A}=(c-1)T+\sum^s_{i=1} k_i \widetilde{E}_i, \quad \psi^{-1}(X)=T+\sum^s_{i=1} a_i \widetilde{E}_i.
\end{equation}

\begin{claim}\label{claim04} Recall that $\widetilde{V}$ is nonsingular locally complete intersection on $\widetilde{A}$ (cf. Claim \ref{claim02}). Then $\widetilde{V}$ has normal crossing with $T, \widetilde{E}_1,\cdots, \widetilde{E}_s$.
\end{claim}

\textit{Proof of Claim \ref{claim04}.} The question is local so we just need to look at local equations. Let $\overline{U}_k=\Spec \overline{R}_k$ be an affine open set of $\overline{A}_k$  such that $I_{\overline{X}_k}=(\overline{f}_1,\cdots,\overline{f}_t)\subset \overline{R}_k$ and $E^k_i$ has a local equation $h_i\in \overline{R}_k$ for $i=1,\cdots, s$ (cf. proof of Claim \ref{Claim01}). Let $\overline{U}=\overline{U}_k\otimes \Spec k[U_{ij}]$ be the corresponding affine open set in $\overline{A}$. Write $\overline{R}=\overline{R}_k\otimes k[U_{ij}]$. Then $I_{\overline{X}}=(\overline{f}_1,\cdots,\overline{f}_t)\cdot \overline{R}$ and each $E_i$ is still defined by the equation $h_i$ in the ring $\overline{R}$. Now let $U_1=\Spec \overline{R}[\overline{f}_2/\overline{f}_1,\cdots, \overline{f}_t/\overline{f}_1]$ be one canonical cover of $\widetilde{A}$ over $U$. Then on $U_1$ the divisor $T$ is defined by the equation $\overline{f}_1$. Notice that each $\widetilde{E}_i$ is still defined by the local equation $h_i\in \overline{R}[\overline{f}_2/\overline{f}_1,\cdots, \overline{f}_t/\overline{f}_1]$. On $U_1$ the variety $\widetilde{V}$ is defined by $I_{\widetilde{V}}=(\widetilde{\alpha}_1,\cdots,\widetilde{\alpha}_c)$, where
$$\widetilde{\alpha}_i=U_{i,1}+U_{i,2}\overline{f}_2/\overline{f}_1+\cdots+U_{i,t}\overline{f}_2/\overline{f}_1,\quad\mbox{for }i=1,\cdots, c.$$
Now we just need to show on $U_1$, $I_{\widetilde{V}}$, $\overline{f}_1$, $h_1,\cdots, h_c$ are normal crossings. Notice that $\overline{f}_1,h_1,\cdots, h_c$ are already normal crossings by the construction and they all sit in the ring $\overline{R}_k$. But $I_{\widetilde{V}}$ is essentially defined by variables $\widetilde{\alpha}_i$'s over $\overline{R}_k$. Thus a local calculation shows that $I_{\widetilde{V}}$ meets $\overline{f}_1,h_1,\cdots, h_c$ as normal crossings. This finishes the proof of Claim \ref{claim04}.\\

Now recall $\nu: A'\longrightarrow A$ is the blowing-up of $\widetilde{A}$ along $\widetilde{V}$ with the exceptional divisor $S$ and $f=\nu\circ\psi:A'\longrightarrow A$. Write $E'_i=\nu^* \widetilde{E}_i$ for $i=1,\cdots, s$ and $T'=\nu^* T$. By Claim \ref{claim04} above we see that $T'$, $E'_1$, $\cdots$, $E'_s$ are all prime divisors and $\Exc(f)=S\cup T'\cup^s_{i=1}E'_i$ are simple normal crossings. Thus $f: A'\longrightarrow A$ is a log resolution of $(A,V+X)$, which we use to compute log canonical thresholds. Notice that $K_{A'/\widetilde{A}}=(c-1)S$. From (\ref{eq15}) we can write
$$K_{A'/A}=(c-1)S+(c-1)T'+\sum^s_{i=1}k_iE', \quad f^{-1}(X)=T'+\sum^s_{i=1}a_iE'_i.$$
Since $I_{V}\cdot \sO_{\widetilde{A}}=I_{\widetilde{V}}\cdot \sO_{\widetilde{A}}(-\psi^{-1}(X))$ (cf. Claim \ref{Claim01} and \ref{claim02}) we then have
$$f^{-1}(V)=S+T'+\sum^s_{i=1}a_iE'_i.$$
Finally by the definition of log canonical threshold we see that
$$\lct(A,X)=\min\{\frac{k_i+1}{a_i}, \frac{(c-1)+1}{1},\frac{(c-1)+1}{0}\}$$
and
$$\lct(A,V)=\min\{\frac{k_i+1}{a_i},\frac{(c-1)+1}{1},\frac{(c-1)+1}{1}\}.$$
Therefore $\lct(A,X)=\lct(A,V)$ as required.
\end{proof}

\begin{corollary}\label{prop08} With notation as in Definition \ref{def01}, if $I_X=\sI(A,cX)$, where $\sI(A,cX)$ is the multiplier ideal sheaf associated to the pair $(A,cX)$, then $$I_V=\sI(A,cV) \mbox{ and }I_Y=\sI(A,cY),$$ where $\sI(A,cV)$ and $\sI(A,cY)$ are multiplier ideal sheaves associated to the pairs $(A,cV)$ and $(A,cY)$, respectively.
\end{corollary}
\begin{proof} By Ein's Lemma, $I_X=\sI(A,cX)$ if and only if $\sI(A,(c-1)X)=\sO_A$. Thus $\lct(A,X)>(c-1)$ and therefore by Theorem \ref{thm01} $\lct(A,Y)\geq\lct(A,V)=\lct(A,X)>(c-1)$. Hence the multiplier ideal sheaves $\sI(A,(c-1)Y)$ and $\sI(A,(c-1)V)$ are all trivial. The result then follows by using Ein's Lemma again.
\end{proof}

\begin{remark} In the above corollary, the equality $I_X=\sI(A,cX)$ is equivalent to the equality $I_{X_k}=\sI(A_k,cX_k)$, where $\sI(A_k,cX_k)$ is the multiplier ideal sheaf associated to the pair $(A_k,cX_k)$. This is because the morphism $A\longrightarrow A_k$ is smooth and the ring extension $R_k\longrightarrow R$ is faithfully flat.
\end{remark}

\begin{corollary}\label{prop09} With notation as in Definition \ref{def01}, let $Y$ be a generic link to an affine pair $(A_k,cX_k)$. Suppose that the pair $(A_k,cX_k)$ is log canonical. Then the pair $(A,cY)$ is also log canonical.
\end{corollary}
\begin{proof} Since, by assumption, $(A_k,cX_k)$ is log canonical and, by Lemma \ref{prop03}, $\sI(A_k,cX_k)\subseteq I_{X_k}$ we see that $\lct(A_k,X_k)=c$. Thus by Theorem \ref{thm01}, we have $\lct(A,Y)\geq c$. But by Lemma \ref{prop03}, we have $\lct(A,Y)\leq c$. Therefore $\lct(A,Y)=c$ and thus $(A,cY)$ is log canonical.
\end{proof}

\begin{proof}[{\bf Proof of Theorem 1.1 in Situation A}]

First of all in this case, $Y$ cannot be empty. The fact that $Y$ is reduced is standard by \cite[2.6]{HunekeUlrich:DivClass}. The rest of the theorem follows from Proposition \ref{prop02}, Proposition \ref{prop04} and Proposition \ref{thm01}.
\end{proof}

\begin{remark}\label{rmk02} Using results established above, we then look at a sequence of generic linkages. Precisely, let $(A_k,cX_k)$ be an affine pair as in Definition \ref{def01} and set $Y_0:=X_k$ and $A_0:=A_k$. We denote a generic link of $Y_0$ to be $Y_1$, which is in a nonsingular ambient space $A_1$. We can continue to construct a generic link of $Y_1$ as $Y_2$ in an nonsingular ambient space $A_2$. Consequently, we get a sequence $Y_0, Y_1, \cdots$, such that each $Y_i$ is a generic link of $Y_{i-1}$ and is in a nonsingular variety $A_i$. Now we list some interesting consequences from the above results as follows.
\begin{itemize}
\item [(1)] If $I_{Y_0}=\sI(A_0,cY_0)$, then $I_{Y_i}=\sI(A_i,cY_i)$, i.e. equality of multiplier ideal with ideal is preserved by generic linkages.
\item [(2)] If $(A_0,cY_0)$ is log canonical then $(A_i,cY_i)$ is log canonical, i.e. log canonical pair is preserved by generic linkage.
\item [(3)] If $(A_0,cY_0)$ is log canonical and $Y_0$ is rational then $(A_i,cY_i)$ is log canonical and $Y_i$ is rational, i.e. log canonical plus rational is preserved by generic linkages.
\item [(4)] $\lct(A_0,Y_0)\leq \lct(A_1,Y_1)\leq \cdots\leq c$, i.e. log canonical thresholds increase in generic linkages but bounded above by $c$.
\end{itemize}
Notice that in (4) we get an increasing but bounded above sequence. Thus it must have a limit, which we denoted by $\lct_{\infty}(A_0,Y_0)$. It would be very interesting to know  if $\lct_{\infty}(A_0,Y_0)$ is independent on the choice of generic linage sequence $Y_1,Y_2,\cdots$. It has been conjectured by the author that after finitely many steps of generic link sequences we will have $\lct_{\infty}(A_0,Y_0)=c$. However, it was pointed out by Bernd Ulrich that the conjecture is not true because otherwise if we start with a Cohen-Macaulay $Y_0$ we will end up with a variety $Y$ which has rational singularities, but it is not the case.
\end{remark}

\begin{remark} There is a conjecture made by the author in \cite[Conjecture 1.4]{Niu:ABoundCMofLC} which asserts that if $X$ is a local complete intersection with log canonical singularities then a generic link $Y$ of $X$ is also a local complete intersection with log canonical singularities. Now it is clear that this conjecture is false. One main reason is that $Y$ cannot be a local complete intersection. However, Corollary \ref{prop09} says that the log canonical pair is preserved by generic linkages and in the conjecture the pair $(A,cX)$ is actually log canonical.
\end{remark}

In the last of this section, we consider specialization problem stated in Situation B. It is a direct consequence of Situation A by restricting to the general fiber over $\Spec k[U_{i,j}]$. This quick proof of the theorem was suggested by the referee.

\begin{proof}[{\bf Proof of Theorem 1.1 in Situation B}] Keep the notation and construction in the proof of Proposition \ref{prop02}. Let $\nA=\Spec k[U_{i,j}]$ be the affine space parameterizing the scalar matrices $(a_{i,j})_{c\times t}$ . Recall that we have the following diagram
\begin{equation}
\xymatrix{
\widetilde{V} \ar@{^{(}->}[r]\ar[d] & \widetilde{A} \ar[d]^{\psi}   \\
Y  \ar@{^{(}->}[r] & A
}\end{equation}
where $\widetilde{V}$ is the embedded resolution of $Y$ (cf. Claim \ref{claim02}).
It is a diagram over $\nA$. Thus it is enough to prove the theorem for a general fiber of this diagram over $\nA$. Since we assume a general fiber of $Y$ over $\nA$ is nonempty the morphism $Y\longrightarrow \nA$ must be dominant. (In fact the only case that a general fiber of $Y$ is empty is that $X_k$ is a complete intersection. In this case we cannot make generic link under Situation B.) Thus the image of $Y$ must contain a open set $U$ of $\nA$. By replacing $\nA$ by this open set we may assume that the morphism $Y\longrightarrow \nA$ is surjective. Now over a  general point $p\in \nA$, the diagram is
\begin{equation*}
\xymatrix{
\widetilde{V}_p \ar@{^{(}->}[r]\ar[d] & \widetilde{A}_p \ar[d]^{\psi_p}   \\
Y_p  \ar@{^{(}->}[r] & A_p
}\end{equation*}
Notice that by the construction $\widetilde{A}_p$ is the blowing-up of $\overline{A}_k$ along $\overline{X}_k$ and the morphism $\psi_p$ is an isomorphism over $A_k\slash X_k$. Also by the generic smoothness $\widetilde{V}_p$ is the embedded resolution of $Y_p$. Thus $Y_p$ is generically reduced. But $Y_p$ does not have embedded components and therefore is reduced. Now the rest of the argument is exact the same as the proof of the theorem in Situation A.
\end{proof}

\begin{remark}\label{rmk03} Corollary \ref{prop07}, \ref{prop08} and  \ref{prop09} are still true in Situation B. However, in Corollary \ref{cor01} and \ref{prop06} the ``if" parts are still true but ``only if" parts are not in general.
\end{remark}

\begin{remark}\label{rmk04} As in Remark \ref{rmk02} (4), we can look at an increasing sequence of log canonical thresholds $\lct(A,Y_0)\leq \lct(A, Y_1)\leq \cdots\leq c$, in which $Y_i$ is a generic link of $Y_{i-1}$ in Situation B. It was pointed out by Lawrence Ein that by using ACC for log canonical thresholds \cite{Ein:ACConSmooth}, after finitely many steps the number $\lct$ becomes stable in the sequence, which we denoted by $\lct_{\infty}(A,Y_0)$. At the moment, it is not clear to us if this number depends on the specific linkage sequence or not. It would be very interesting to have a further investigation on this number. We hope that this new invariant will have some application in linkage classes of varieties.
\end{remark}

\section{Generic linkage of projective varieties and Castelnuovo-Mumford regularity}

\noindent In this section, we study a generic link of a subvariety of $\nP^n$ in Situation C.  The main idea is inspired by the work of Betram, Ein and Lazarsfeld \cite{BEL}. Thus we shall be brief in proofs. Throughout this section, we assume that $A$ is a projective nonsingular variety over $k$ and  $L$ is a line bundle on $A$ generated by its global sections. We shall prove the following theorem which can be applied to a slightly more general case than Situation C.

\begin{theorem}\label{prop01} Suppose that $X\subset A$ is a reduced equidimensional subscheme of codimension $c$ scheme-theoretically defined by the $t$ sections $s_i\in H^0(A, L^{d_i})$ with $d_1\geq d_2\geq\cdots\geq d_t$. Take $c$ general sections $\alpha_i\in H^0(A,\sI_X\otimes L^{d_i})$ for $i=1,\cdots, c$ and let $V$ be a complete intersection defined by the vanishing of $\alpha_1,\cdots,\alpha_c$. Let $Y$ be a subscheme of $A$ defined by  $\sI_Y:=(\sI_V:\sI_X)$. Then either $Y$ is empty or else:
\begin{itemize}
\item [(1)] $Y$ is reduced equidimensional of codimension $c$ (possibly reducible) geometrically linked to $X$ via $V$.
\item [(2)] $\omega^{GR}_Y\simeq\sI(A,cX)\cdot \sO_Y\otimes\omega_A \otimes L^{d_1+\cdots+d_c} $,
where $\sI(A,cX)$ is the multiplier ideal sheaf associated to the pair $(A,cX)$, and it fits into the following commutative diagram
    $$\xymatrix{
\omega^{GR}_Y \ar[r]^-{\simeq}\ar@{^{(}->}[d]_{\tr} & {\sI(A,cX)\cdot \sO_Y\otimes \omega_A\otimes L^{d_1+\cdots+d_c}} \ar@{^{(}->}[d]^i   \\
\omega_Y \ar[r]^-{\simeq}  & \sI_X\cdot\sO_Y\otimes \omega_A\otimes L^{d_1+\cdots+d_c}
}$$
where the bottom one is given by \ref{eq04}, $\tr$ is the trace map, and the inclusion $i$ is induced by $\sI(A,cX)\hookrightarrow \sI_X$ (cf. Lemma \ref{prop03}).
\item [(3)] $\lct(A,Y)\geq \lct(A,V)=\lct(A,X)$.
\end{itemize}
\end{theorem}
\begin{proof} Take a log resolution of singularities for the pair $(A,X)$ as
$f:\overline{A}\longrightarrow A$
such that $\sI_X\cdot\sO_{\overline{A}}=\sO_{\overline{A}}(-E)$ where $E$ is an effective divisor and $\Exc(f)\cup E$ is a divisor with simple normal crossing support. We may also assume that the morphism $f$ is an isomorphism over the open set $A\backslash X$.

For $i=1,\cdots, t$ we denote by $\mf{b}_i$ the sub-linear system of $|L^{d_i}|$ determined by the vector space $H^0(A, \sI_X\otimes L^{d_i})$. We use notation $(s_i)_0$ to be the effective divisor in the linear system $\mf{b}_i$ defined by the zero locus of the section $s_i$. Since $X$ is defined by the vanishing of sections $s_i$ we have a surjective morphism
$\oplus^tL^{-d_i}\xrightarrow{(s_1,\cdots,s_t)}\sI_X\longrightarrow 0$.
Pulling back this surjective morphism via $f$, we then obtain a surjective morphism
\begin{equation}\label{eq03}
\bigoplus^tf^*L^{-d_i}\xrightarrow{(f^*s_1,\cdots,f^*s_t)}\sO_{\overline{A}}(-E)\longrightarrow 0.
\end{equation}
(Since $f$ is dominant the induced morphism on the linear system $f^*:\mf{b}_i\longrightarrow |f^*L^{d_i}|$ is actually injective. So we think of $\mf{b}_i$ naturally as a sub-linear system of $|f^*L^{d_i}|$ under $f^*$.)
Denote by $\mf{B}_i=f^*\mf{b_i}-E$ the sub-linear system of $|f^*L^{d_i}(-E)|$ obtained from $f^*\mf{b}_i$ by removing the base locus divisor $E$.  Then the section $f^*s_i$ naturally gives rise to a section $\sigma_i$ of $f^*L^{d_i}(-E)$ defining an effective divisor $F_i$ in the linear system $\mf{B}_i$. Thus from the surjectivity of (\ref{eq03}), we deduce a surjection
$$\bigoplus^tf^*L^{-d_i}(E)\xrightarrow{(\sigma_1,\cdots,\sigma_t)}\sO_{\overline{A}}\longrightarrow 0,$$
and therefore we have
\begin{equation}\label{eq02}
F_1\cap F_2\cap\cdots\cap F_t=\phi.
\end{equation}
We make the following observation.

\begin{claim}\label{claim05}
For each $i=1,\cdots, t$, one has
\begin{itemize}
\item [(a)] The system $\mf{B}_1$ is base point free.
\item [(b)] For each $i\geq 1$, the base locus $\Bs(\mf{B}_i)$ of the system $\mf{B}_i$ is inside the support of the divisor $F_j$ for $j\geq i$.
\end{itemize}
\end{claim}

\textit{Proof of Claim \ref{claim05}.} For the statement (a), by the definition of $X$, we see that the sheaf $\sI_X\otimes L^{d_1}$ is generated by its global sections. Let $W_1=H^0(A,\sI_X\otimes L^{d_1})$ so that $\mf{b}_1=|W|$. So we have a surjective morphism $W_1\otimes L^{d_1}\longrightarrow \sI_X\longrightarrow 0$. Thus we have a surjective morphism $W_1\otimes f^*L^{-d_1}\longrightarrow \sO_{\overline{A}}(-E)\longrightarrow 0$, i.e., $W_1\otimes f^*L^{-d_1}(E)\longrightarrow \sO_{\overline{A}}\longrightarrow 0$. Thus we see that $\mf{B}_1$ is base point free.

For the statement (b), notice that when $i=1$ the result is trivial from (a). We prove the first nontrivial case when $i=2$. It is from the definition of base loci that $\Bs(\mf{B}_2)\subset F_2$. Now we show $\Bs(\mf{B}_2)\subset F_3$. Denote by $\delta^2_3$ the linear system $|L^{d_2-d_3}|$ which is base point free. Notice that we have an inclusion $\delta^2_3+(s_3)_0\subset \mf{b}_2$. Thus pulling back via $f$, we have the inclusion $f^*\delta^2_3+f^*(s_3)_0\subset f^*\mf{b}_2$ and therefore by subtracting the divisor $E$ we see $f^*\delta^2_3+f^*(s_3)_0-E\subset f^*\mf{b}_2-E$. Recall that $F_3=f^*(s_3)_0-E$ and $\mf{B}_2=f^*\mf{b}_2-E$. Thus we have an inclusion
$$f^*\delta^2_3+F_3\subset \mf{B}_2.$$
From this, the linear system $\mf{B}_2-F_3$ contains the system $f^*\delta^2_3$, which is base point free. Thus the base locus $\Bs(\mf{B}_2)$ is contained in $F_3$. Similar argument works for all $j\geq i$, which proves the Claim \ref{claim05}.\\

Now since $\mf{B}_1$ is base point free, by Bertini's theorem (Cf. \cite[Corollary III.10.9]{Har:AG}) we can take a general element $D_1\in \mf{B}_1$ such that (i) $D_1$ is nonsingular and equidimensional; (ii) no components of $D_1$ are contained in the support of $E\cup \Exc(f)$; (iii) $D_1$ has normal crossing with $\Exc(f)$; (iv) $D_1\cap F_2\cap\cdots\cap F_t=\phi$. Here the reason for (iv) is that since the section $\sigma_1$ is nowhere vanishing on $F_2\cap\cdots\cap F_t$, the general section of $\mf{B}_1$ is then nowhere vanishing on $F_2\cap\cdots\cap F_t$.

Now by Claim \ref{claim05} the base locus $\Bs(\mf{B}_2)$ is inside $F_j$ for $j\geq 2$. Thus we have $\Bs(\mf{B}_2)\subset F_2\cap F_3\cap\cdots\cap F_t$. By the fact (\ref{eq02}) and the choice of $D_1$, the linear system $\mf{B}_2$ is base point free on $D_1$. Thus by Bertini's theorem again we can choose a general element $D_2\in \mf{B}_2$ such that (i) $D_1\cap D_2$ is nonsingular and equidimensional; (ii) no components of $D_1\cap D_2$ are contained in the support of $E\cup \Exc(f)$; (iii) $D_1\cap D_2$ has normal crossing with $\Exc(f)$; (iv) $D_1\cap D_2\cap F_3\cap\cdots\cap F_t=\phi$.

We then can iterate such argument by $c$ times to end up with a subscheme $$\overline{Y}:=D_1\cap D_2 \cap\cdots \cap D_c$$ of $\overline{A}$ such that $\overline{Y}$ is either empty or else (i) $\overline{Y}$ is nonsingular and equidimensional; (ii) no component of $\overline{Y}$ is contained in the support of $E\cup \Exc(f)$; (iii) $\overline{Y}$ has normal crossing with $\Exc(f)$; (iv) $\overline{Y}\cap F_{c+1}\cap\cdots\cap F_t=\phi$.
Notice that each effective divisor $D_i$ is a general element in the linear system $\mf{B}_i$.

Now each $D_i$ naturally corresponds to a general section $\alpha_i\in H^0(A,\sI_X\otimes L^{d_i})$. (Recall that the divisor $D_i+E$ is in $f^*\mf{b}_i$.) Those $\alpha_1,\cdots, \alpha_c$ cut out a complete intersection $V$ on $A$. Let $Y$ be the subscheme of $A$ linked to $X$ via $V$, i.e., $Y$ is defined by an ideal sheaf $\sI_Y:=(\sI_V:\sI_X)$. It is well-know that $Y$ is equidimensional of codimension $c$ without embedded components and with no common components with $X$. Notice that at least set-theoretically $Y=f(\overline{Y})$. Recall that the morphism $f$ is an isomorphism over the open set $A\backslash X$. Thus by the construction of $\overline{Y}$ we see that $\overline{Y}\cap f^{-1}(U)$ is isomorphic to $Y\cap U$. Therefore $f$ is an isomorphism at the generic points of $\overline{Y}$ and then $\overline{Y}$ is the strict transform of $Y$. Thus $Y$ is generically smooth and therefore is reduced. The rest of the statement (1) are all standard result, so we would not repeat here.

By the construction of $\overline{Y}$, we have a surjective morphism
$\oplus^cf^*L^{-d_i}(E)\longrightarrow \sI_{\overline{Y}}\longrightarrow 0$. Thus the normal bundle of $\overline{Y}$ inside $\overline{A}$ can be easily calculated and the rest argument for (2) is exactly the same as the proof of Propositions \ref{prop02} and \ref{prop04}. The reducedness of $Y$ in (1) is also clear since $Y$ is generic reduced and has no any embedded components. Finally, the statement (3) has the same proof as in Proposition \ref{thm01} so we would not repeat here.
\end{proof}

\begin{remark} If $A=\nP^n$, we can actually reduce Situation C to Situation B, at least in the case that all equations have the same degree. This was kindly suggested by the referee. Here let us explain this reduction briefly. We choose general homogeneous coordinates for $\nP^n$ so that all the equations $f_i$'s restricted to the same degree polynomials, say $g_i$'s, on the standard charts. Then we choose a general matrix $B=(b_{ij})_{c\times t}$ and we can find a invertible lower triangular $c\times c$ matrix $A$ so that $A\cdot B=(B_1|B_2)$, where $B_1$ is a $c\times c$ upper triangular matrix. Now the complete intersection $(\beta_1,\cdots,\beta_c) =A\cdot B\cdot (g_1,\cdots,g_t)^T$ will give back general sections $\alpha_i$ in $H^0(\nP^n,\sI_X(d_i))$.
\end{remark}

As an application of above theorem, using an idea of \cite{CU:Reg} we can generalize results of de Fernex and Ein \cite[Corollary 1.4]{Ein:VanishLCPairs} and Chardin and Ulrich \cite[Theorem 0.1]{CU:Reg} on the Castelnuovo-Mumford regularity bound. The bound was first established in \cite{BEL} for nonsingular case and then generalized in \cite{CU:Reg} for rational singular case. \cite{BEL} and \cite{CU:Reg} allow variety to have some not too bad very singular loci but \cite{Ein:VanishLCPairs} cannot allow this situation, which we are able to handle now (cf. \cite[Remark 5.2]{Ein:VanishLCPairs}). Recall that a coherent sheaf $\sF$ on the projecive space $\nP^n$ is said to be $m$-regular if $H^i(\nP^n, \sF(m-i))=0$ for all $i>0$. The minimal such number $m$ is called the {\em Castelnuovo-Mumford regularity} of $\sF$ and is denoted by $\reg \sF$. If $X$ is a subscheme of $\nP^n$ defined by an ideal sheaf $\sI_X$ then the regularity of $X$ is defined to be $\reg X=\reg\sI_X$.

\begin{corollary}\label{cor03} Let $X\subset \nP^n$ be a reduced equidimensional subscheme of codimension $c$ defined by the equations of degree $d_1\geq d_2\geq \cdots\geq d_t$ and let $Z=\Supp (\sI_X/\sI(\nP^n,cX))$ where $\sI(\nP^n,cX)$ is the multiplier ideal sheaf. Assume that $\dim Z\leq 1$ and each one dimensional irreducible component of $Z$ has at least one point at which $X$ is a local complete intersection.  Then
$$\reg X\leq \sum^c_{i=1}d_i-c+1$$
and equality holds if and only if $X$ is a complete intersection in $\nP^n$.
\end{corollary}

\begin{proof}
If $X$ is a complete intersection then the result is clear. So in the sequel we assume $X$ is not a complete intersection. Then take a generic link $Y$ of $X$ via a complete intersection $V$ cut out by the general equations of $X$ of degree $d_1\geq d_2\geq \cdots \geq d_c$. By the assumption on $Z$ we can choose $V$ general such that $Y\cap  Z$ has dimension $\leq 0$, i.e. no any dimension one irreducible component of $Z$ can be contained in $Y$. This means that from Theorem \ref{prop01} we have the following exact sequence
\begin{equation}\label{eq18}
0\longrightarrow \omega^{GR}_Y\longrightarrow \omega_Y\longrightarrow \sQ\longrightarrow 0
\end{equation}
where $\dim \Supp\sQ\leq 0$.

Denote by $d:=\sum^c_{j=1}d_j$ and $r:=\dim V=\dim X=\dim Y$.
It is clear that from the Kawamata-Viehweg vanishing theorem we can deduce $\reg \omega^{GR}_Y = r+1$ and therefore $\reg \omega_Y=r+1$ by the sequence (\ref{eq18}). Now recall $\omega_Y=\sI_X\cdot\sO_V\otimes \omega_V$. Since $V$ is a complete intersection we see $\reg \I_V=d-c+1$ and $\omega_V=\sO_V(d-n-1)$ by Koszul resolution and thus $\omega_Y=\sI_X\cdot\sO_V(d-n-1)$. Now from the exact sequence
$$0\longrightarrow \sI_V\longrightarrow \sI_X\longrightarrow \sI_X\cdot\sO_V\longrightarrow 0,$$
it is immediately that $\reg \sI_X\leq d-c$.
\end{proof}

\begin{remark} In the corollary, by Ein's Lemma the set $Z$ can be equivalently defined as $$Z:=\{x\in \nP^n\ |\ \sI(\nP^n, (c-1)X) \mbox{ is non trivial at } x\}.$$
The condition on $Z$ in the corollary already includes conditions discussed in de Fernex and Ein \cite{Ein:VanishLCPairs} and Chardin and Ulrich \cite{CU:Reg} and is certainly more general. It is worth mentioning that the above liaison method to bound the regularity also can be used directly to study so called  multiregularity once the correct corresponding form of generic linkage can be built.
\end{remark}

In Theorem \ref{prop01}, a generic link $Y$ is usually not necessarily irreducible. One important case in applications is when $X$ is cut out by sections of the same degree, i.e., $d_1=\cdots=d_t$. If the defining ideal sheaf $\sI_X$ has enough sections then it is possible that a generic link $Y$ is irreducible. One way to see this is by using the $s$-invariant of $\sI_X$ with respect to $L$, which measures the positivity of $\sI_X$. We recall its definition and refer to \cite{Lazarsfeld:PosAG1} for further reference.

\begin{definition} Given an ideal sheaf $\sI$ on $A$ let
$\mu:W=\Bl_{\sI}A\longrightarrow A$ be the blowing-up
of $A$ along the ideal $\sI$ with an exceptional Cartier
divisor $E$ on $W$, such that $\sI\cdot\sO_W=\sO_W(-E)$. Let $L$ be
an ample line bundle on $A$. We define the {\em $s$-invariant} of
$\sI$ with respect to $L$ to be the positive real number
$$s_L(\sI):=\min\{\ s\ |\ s\mu^*L-E \ \mbox{ is nef }\}.$$
Here $ s\mu^*L-E$ is considered as an $\nR$-divisor on $W$.
\end{definition}

\begin{remark} For example, in the projective space case the line bundle is usually taken to be the hyperplane divisor and simply write the $s$-invariant as $s(\sI)$. In this case suppose that $d$ is an integer such that $\sI(d)$ is generated by global sections, then it is easy to see that $s(\sI)\leq d$, i.e., the $s$-invariant of $\sI$ is always bounded by a generating degree of $\sI$. There is an example in \cite{Ein:PosiComlIdeal} showing that the $s$-invariant could be an irrational number.
\end{remark}

\begin{corollary} With notation and assumption as in Theorem \ref{prop01} assume further that $L$ is ample and $X$ is cut out by the sections of the same degree, i.e., $d_1=\cdots=d_t=d$. If the $s$-invariant of $\sI_X$ with respect to $L$ is strictly smaller than $d$, then $Y$ is nonempty and irreducible.
\end{corollary}
\begin{proof} We keep notation as in the proof of Theorem \ref{prop01} but set $d_1=\cdots=d_t=d$. Now this time the linear systems $\mf{B}_1,\cdots, \mf{B}_t$ are all the same as the linear system $\mf{B}:=f^*{\mf{b}}-E$, where $\mf{b}$ is the sub-linear system of $|L^d|$ determined by the vector space $H^0(A,\sI_X\otimes L^d)$. Notice that $\mf{B}$ is a sub-linear system of $|f^*L^d(-E)|$ and is base point free. Also notice that $\dim \mf{B}=\dim \mf{b}$ since $f$ is dominant and $E$ is the basic locus of $f^*\mf{b}$.

All we need is to show the subscheme $\overline{Y}$, which is the intersection of $c$ general divisors $D_i\in \mf{B}$, is irreducible. Recall that $\mf{B}$ is base point free, it then gives a morphism to a projective space
$$\phi_{\mf{B}}:\overline{A}\longrightarrow \nP^r,$$
where $r=\dim \mf{B}$ such that $f^*L^d(-E)=\phi^*_{\mf{B}}\sO_{\nP^r}(1)$.
We claim that the morphism $\phi_{\mf{B}}$ is generically finite. To see this, let $\mu: A'=\Bl_XA\longrightarrow A$ be the blowing up of $A$ along $X$ with an exceptional divisor $F$ such that $\sI_X\cdot\sO_{A'}=\sO_{A'}(-F)$. By the universal property of blowing-ups we have $g:\overline{A}\longrightarrow A'$ such that $f=g\circ\mu $ and $g^*F=E$. Notice that $g$ is generically finite. Now the system $\mf{B}':=\mu^*\mf{b}-F$ is base point free and then gives a morphism $\phi_{\mf{B}'}$to $\nP^r$, which commutes with $\phi_{\mf{B}}$, i.e. $\phi_{\mf{B}}=\phi_{\mf{B'}}\circ g$
$$\xymatrix{
\overline{A} \ar[rr]^{\phi_{\mf{B}}}\ar[dr]_{g} &   & \nP^r  \\
                         & A' \ar[ur]_{\phi_{\mf{B'}}} &
}$$
Since the $s$-invariant of $\sI_X$ with respect to $L$ is strictly smaller than $d$ the line bundle $\mu^*L^d(-F)$ is ample. But $\mu^*L^d(-F)=\phi_{\mf{B}'}^*\sO_{\nP^r}(1)$ so the morphism $\phi_{\mf{B'}}$ is finite. Thus $\phi_{\mf{B}}$ is generically finite and we have $\dim \phi_{\mf{B}}(\overline{A})= \dim \overline{A}$. Now by the theorem of \cite[3.3.1]{Lazarsfeld:PosAG1}, the subscheme $\overline{Y}$ is nonempty and irreducible since $\dim \phi_{\mf{B}}(\overline{A})>c$. Therefore the generic link $Y$ to $X$ is also nonempty and irreducible.
\end{proof}

\begin{remark} Having Theorem \ref{prop01} in hand, it is then easy to get those similar corollaries mentioned in the previous sections (cf. Remark \ref{rmk02}, \ref{rmk03} and \ref{rmk04}). So we leave them to the reader.
\end{remark}

\bibliographystyle{alpha}

\end{document}